\pdfoutput=1
\documentclass[11pt,oneside,reqno]{amsart}
\usepackage{amsmath, amssymb, amsthm, amsfonts, mathrsfs}

\setbox0=\hbox{$+$}
\newdimen\plusheight

\plusheight=\ht0
\def\+{\;\lower\plusheight\hbox{\underline{}$+$}\;}

\setbox0=\hbox{$-$}
\newdimen\minusheight
\minusheight=\ht0
\def\-{\;\lower\minusheight\hbox{$-$}\;}

\setbox0=\hbox{$\cdots$}
\newdimen\cdotsheight
\cdotsheight=\plusheight
\def\cds{\lower\cdotsheight\hbox{$\cdots$}}

\newtheorem{rmk}{Remark}
\newtheorem{question}{Question}

\newtheorem{conjecture}{Conjecture}

\theoremstyle{definition}

\newtheorem{example}{Example}
\usepackage{graphicx, epsfig}
\theoremstyle{definition}
\makeatletter

\usepackage[strict]{changepage}
\makeatother

\numberwithin{equation}{section}
\theoremstyle{plain}

\newtheorem{theorem}{Theorem}[section]
\newtheorem{corollary}[theorem]{Corollary}

\newtheorem{definition}[theorem]{Definition}
\newtheorem{lemma}[theorem]{Lemma}
\newtheorem{claim}[theorem]{Claim}

\usepackage{titlesec}

\usepackage{graphicx, epsfig}
\theoremstyle{definition}
\usepackage{graphicx} 
\usepackage{color} 
\usepackage{transparent}

\titleformat*{\section}{\LARGE\bfseries}

\makeatletter
\renewcommand\section{\@startsection{section}{1}{\z@}%
                                  {-3.5ex \@plus -1ex \@minus -.2ex}%
                                  {2.3ex \@plus.2ex}%
                                  {\normalfont\large\bfseries}}
\makeatother
\usepackage{graphicx} 
\makeatother
\usepackage{graphicx} 
\usepackage{color} 
\usepackage{transparent} 
\usepackage{geometry}
\title{On the number of hyperbolic Dehn fillings of a given volume}
\author{BoGwang Jeon}

\begin{document}
\maketitle
\begin{abstract}
Let $\mathcal{M}$ be a $1$-cusped hyperbolic $3$-manifold whose cusp shape is quadratic. We show that there exists $c=c(\mathcal{M})$ such that the number of hyperbolic Dehn fillings of $\mathcal{M}$ with any given volume $v$ is uniformly bounded by $c$. 
\end{abstract}

\section{Introduction}
\subsection{Main theorem}
The following theorem due to W. Thurston is well-known:
\begin{theorem}[W. Thurston]\label{18110701}
Let $\mathcal{M}$ be an $n$-cupsed hyperbolic $3$-manifold and $M_{(p_1/q_1,\dots, p_n/q_n)}$ be its $(p_1/q_1,\dots, p_n/q_n)$-Dehn filling. Then 
\begin{equation*}
\text{\normalfont\ vol}(\mathcal{M}_{(p_1/q_1,\dots, p_n/q_n)})\rightarrow \text{\normalfont\ vol}(\mathcal{M})
\end{equation*}
as $|p_i|+|q_i|\rightarrow \infty$ ($1\leq i\leq n$).
\end{theorem}
This implies 
\begin{corollary}\label{18110702}
Let $\mathcal{M}$ be an $n$-cusped hyperbolic $3$-manifold. The number $N_{\mathcal{M}}(v)$ of hyperbolic Dehn fillings of $\mathcal{M}$ with a given volume $v$ is finite. 
\end{corollary}
Combining with the work of J$\o$rgensen-Thurston, Corollary \ref{18110702} implies 
\begin{corollary}\label{18110703}
The number $N(v)$ of hyperbolic $3$-manifolds with a given volume $v$ is finite.  
\end{corollary}
Thus it is natural to study the behavior of $N(v)$ or $N_{\mathcal{M}}(v)$ as $v$ varies.

The existence of $v$ satisfying $N(v)=1$ or $N_{\mathcal{M}}(v)=1$ has been known by various works. Gabai-Meyerhoff-Milley \cite{GMM1, GMM2, M} showed that the Weeks manifold with volume $v_1 = 0.9427 \dots$ is the unique hyperbolic $3$-manifold of lowest volume (so $N(v_1) = 1$), and Hodgson-Masai recently found a $1$-cusped hyperbolic $3$-manifold $\mathcal{M}$ and infinitely many $v_i$ such that $N_{\mathcal{M}}(v_i)=1$ \cite{HM}. They further proved that these $v_i$ satisfy $N(v_i)=1$. 



On the other hand, the lim superior of $N(v)$ is known to be unbounded as $v$ goes to infinity. First, Wielenberg \cite{W} showed, for a fixed $n\in \mathbb{N}$, there exists $v$ such that $N(v) > n$, and Zimmermann \cite{Z} proved the same result with closed hyperbolic $3$-manifolds. In general, by considering covering spaces of a fixed hyperbolic 3-manifold whose fundamental group surjects onto a free group of rank $2$, one can construct a sequence of volumes $v_i \rightarrow \infty$ and a constant $c > 0$ such that $N(v_i) > {v_i}^{cv_i}$ for all $i$ \cite{BGLM, C}. 

However it has been unknown whether, for a given cusped hyperbolic $3$-manifold $\mathcal{M}$, the limit superior of $N_{\mathcal{M}}(v)$ possess the same property as $v$ tends to $\text{vol}(\mathcal{M})$. In other words, the following question is still open \cite{gromov, HM}:
\begin{question}\label{18102801}
Let $\mathcal{M}$ be an $n$-cusped hyperbolic $3$-manifold. Does there exist $c=c(\mathcal{M})$ such that $N_{\mathcal{M}}(v)<c$ for any $v$?
\end{question}
As pointed out in \cite{HM}, an affirmative answer to this question will settle the following question, which was originally raised by Gromov in \cite{gromov}:
\begin{question}[Gromov]\label{gromov}
Is $N(v)$ locally bounded?
\end{question} 
In this paper, we answer Question \ref{18102801} partially as follows:
\begin{theorem}\label{18110705}
Let $\mathcal{M}$ be a $1$-cusped hyperbolic $3$-manifold whose cusp shape is quadratic. Then there exists $c$ such that $N_{\mathcal{M}}(v)<c$ for any $v$. 
\end{theorem}

\subsection{An attempt toward counter examples}
One major difficulty in answering Question \ref{18102801} is that one can construct a candidate counter example easily. 

By the work of Neumann-Zagier, we have the following theorem \cite{NZ} (see Theorem \ref{NZ1} for the precise formula):
\begin{theorem}[Neumann-Zagier]\label{18110801}
Let $\mathcal{M}$ be a $1$-cusped hyperbolic $3$-manifold and $\mathcal{M}_{p/q}$ be its $p/q$-Dehn filling. Then there exists a quadratic form $Q_{\mathcal{M}}(p,q)$ such that 
\begin{equation}\label{18111601}
\text{\normalfont\ vol}(\mathcal{M}_{p/q})=\text{\normalfont\ vol}(\mathcal{M})+\dfrac{1}{Q_{\mathcal{M}}(p,q)}+O\Big(\dfrac{1}{Q_{\mathcal{M}}(p,q)^2}\Big).
\end{equation}
\end{theorem}

Suppose there exists $1$-cusped hyperbolic $3$-manifold $\mathcal{M}$ whose volume formula is simply given as 
\begin{equation}\label{18093001}
\text{vol}(\mathcal{M}_{p/q})=\text{vol}(\mathcal{M})-\dfrac{1}{p^2+q^2}, 
\end{equation}
(or, more generally,  
\begin{equation}\label{18103101}
\text{vol}(\mathcal{M}_{p/q})=\text{vol}(\mathcal{M})+h\Big(\dfrac{1}{p^2+q^2}\Big) 
\end{equation}
where $h$ is an analytic function). Then finding $\mathcal{M}_{p/q}$ and $\mathcal{M}_{p'/q'}$ satisfying
\begin{equation*}
\text{vol}(\mathcal{M}_{p/q})=\text{vol}(\mathcal{M}_{p'/q'})
\end{equation*}
is equivalent to finding $(p,q)$ and $(p',q')$ in $\mathbb{Z}^2$ satisfying
\begin{equation*}
\dfrac{c}{p^2+q^2}=\dfrac{c}{(p')^2+(q')^2}\quad \Longleftrightarrow \quad p^2+q^2=(p')^2+(q')^2.
\end{equation*}
It is well-known in number theory that 
\begin{equation*}
\limsup_{r\rightarrow \infty} N_{S_r}(\mathbb{Z})=\infty
\end{equation*}
where $N_{S_r}(\mathbb{Z})$ is the number of lattice points on the circle of radius $r$. Thus the answer to Question \ref{18102801} is no for this example. 

However we prove in Section \ref{Preliminary} that the above types of candidate counter examples never arise in reality. That is, we prove the following lemma:
\begin{lemma}\label{18112502}
Let $\mathcal{M}$ be a $1$-cusped hyperbolic $3$-manifold whose cusp shape is quadratic. Then there is no analytic function $h(t)$ such that 
\begin{equation}
\text{\normalfont\ vol}(\mathcal{M}_{p/q})=\text{\normalfont\ vol}(\mathcal{M})+h\Big(\dfrac{1}{Q_{\mathcal{M}}(p,q)}\Big).
\end{equation}
\end{lemma}

\subsection{Proof highlights}\label{Ph}
The proof of Theorem \ref{18110705} uses a hybrid of several different ideas from number theory, algebraic geometry, and model theory. 

First denote
\begin{equation*}
\text{vol}(\mathcal{M}_{p/q})-\text{vol}(\mathcal{M})
\end{equation*}
by $\Theta_{\mathcal{M}}(p,q)$. We extend its domain and consider it as an analytic function over $\mathbb{R}^2(:=(x,y))$ with sufficiently large $|x|+|y|$. In Lemma \ref{18110304}, it will be shown there exists $m=m(\mathcal{M})\in \mathbb{N}$ such that if 
\begin{equation}\label{18120901}
\text{vol}(\mathcal{M}_{p/q})=\text{vol}(\mathcal{M}_{p'/q'})\quad (\text{or equivalently}\quad \Theta_{\mathcal{M}}(p,q)=\Theta_{\mathcal{M}}(p',q')),
\end{equation}
then\footnote{Here $Q_{\mathcal{M}}(p,q)$ is the same as the one in Theorem \ref{18110801} and so 
\begin{equation*}
\Theta_{\mathcal{M}}(p,q)=\dfrac{1}{Q_{\mathcal{M}}(p,q)}+O\Big(\dfrac{1}{Q_{\mathcal{M}}(p,q)^2}\Big).
\end{equation*}
}
\begin{equation*}
|Q_{\mathcal{M}}(p,q)-Q_{\mathcal{M}}(p',q')|<m.
\end{equation*}
If the cusp shape of $\mathcal{M}$ is quadratic, then the coefficients of $Q_{\mathcal{M}}$ are all rational and the image of  
\begin{equation*}
Q_{\mathcal{M}}(p,q)-Q_{\mathcal{M}}(p',q')
\end{equation*}
is discrete in $(-m,m)$. So $(p,q,p',q')\in \mathbb{R}^4(:=(x,y,z,w))$ satisfying \eqref{18120901} is an integeral point lying over the intersection between 
\begin{equation*}
\Theta_{\mathcal{M}}(x,y)=\Theta_{\mathcal{M}}(z,w)
\end{equation*}
and 
\begin{equation}\label{19032701}
\dfrac{1}{Q_{\mathcal{M}}(x,y)}=\dfrac{1}{Q_{\mathcal{M}}(z,w)+k}
\end{equation}
where $k\in S$ for some finite subset $S$ of $\mathbb{Q}$.\footnote{On the other hand, if the cusp shape of $\mathcal{M}$ is non-quadratic, then $Q_{\mathcal{M}}$ is a quadratic form with irrational coefficients and so, by Oppenheim's conjecture (or, equivalently, Margulis's theorem), the image of 
\begin{equation*}
Q_{\mathcal{M}}(x,y)-Q_{\mathcal{M}}(z,w)\quad (x,y,z,w\in \mathbb{Z})
\end{equation*}
is dense in $\mathbb{R}$. The difference between these two is crucial and that is why we need the assumption ``cusp shape being quadratic" in the statement of the main theorem. That is, it enables us to work with only a finite number of algebraic varieties defined of the form \eqref{19032701}.}

Let  $\mathcal{Z}^k(\subset \mathbb{R}^2\times \mathbb{R}^2)$ be  
\begin{equation}\label{19032802}
\begin{gathered}
\Big(\Theta_{\mathcal{M}}(x,y)=\Theta_{\mathcal{M}}(z,w)\Big)\bigcap\Big(\dfrac{1}{Q_{\mathcal{M}}(x,y)}=\dfrac{1}{Q_{\mathcal{M}}(z,w)+k}
\Big)
\end{gathered}
\end{equation}
and $\pi_1$ (resp. $\pi_2$) be the projection from $\mathbb{R}^4(:=(x,y,z,w))$ onto the first $\mathbb{R}^2$ (resp. second $\mathbb{R}^2$). For $(x,y)\in \pi_1(\mathcal{Z}^k)$, put     
\begin{equation}\label{19032803}
\mathcal{Z}^k_{(x,y)}:=\{(z,w)\in \pi_2(\mathcal{Z}^k)\;:\;(x,y,z,w)\in \mathcal{Z}^k\}. 
\end{equation}
Then, for a given $\mathcal{M}_{p/q}$, the number of hyperbolic Dehn fillings of $\mathcal{M}$ with volume $\text{vol}(\mathcal{M}_{p/q})$ is equal to the number of primitive integer points on the union of the fibers
\begin{equation}\label{18120701}
\bigcup_{-m< k< m} \mathcal{Z}^k_{(p,q)}. 
\end{equation}
Hence, to prove the theorem, it is enough to show the existence of a constant $c=c(\mathcal{M})$ satisfying 
\begin{equation}\label{18120701}
\sum_{-m<k<m} \big|\{(x,y)\in \mathcal{Z}^k_{(p,q)}\;:\;(x,y)\in \mathbb{Z}^2\}\big|<c
\end{equation}
for every $(p,q)$. However since both $\mathcal{Z}^k$ and its fiber $\mathcal{Z}^k_{(x,y)}$ are (presumably) transcendental not algebraic, usual well-known techniques for counting rational or integral points over algebraic varieties in diophantine geometry do not work in this case.  

To overcome this technical difficulty, we employ results from the so called ``o-minimal theory" in logic.  An o-minimal structure is simply a generalization of the class of semialgebraic sets, and an element of it, called a definable set, shares many common properties with a semialgebraic set. For instance, any definable set has only finitely many connected components and the number of connected components of each fiber of it is known to be uniformly bounded. (So if the dimension of $\mathcal{Z}^k_{(p,q)}$ is zero for any $k$ and $(p,q)$, then one immediately gets the desired conclusion.) 

Appealing to the properties of o-minimality would be the technical heart of the paper. We believe the concepts and tools introduced here will be potentially useful to fully answer Question \ref{18102801}. (See Section \ref{18110807}.)


\subsection{Outline of the paper}
In Section \ref{back}, we review some necessary background both on hyperbolic and o-minimal geometries. Basic definitions and theorems that will be used throughout the paper are stated in this section. In Section \ref{Preliminary}, we prove some preliminary lemmas (including Lemma \ref{18112502}) that are required to prove the main theorem. The proofs in Section \ref{Preliminary} are based on the lemmas and theorems discussed in  Section \ref{back}. Section \ref{proof} is the core of the paper and contains the proof of the main theorem. Finally, in Section \ref{18110807}, we propose several conjectures concerning the main question. 

\subsection{Acknowledgement}

The author would like to thank the referee for his/her careful reading of the paper as well as many helpful comments and suggestions. 

\section{Background}\label{back}
\subsection{Neumann-Zagier asymptotic volume formula}\label{NZAVF}
For a $1$-cusped hyperbolic $3$-manifold $\mathcal{M}$, let  
\begin{equation}\label{18110307}
\text{Hom}\;\big(\pi_1(\mathcal{M}), PSL_2(\mathbb{C})\big)/\sim
\end{equation}
be the quotient space of $\text{Hom}\;\big(\pi_1(\mathcal{M}),  PSL_2(\mathbb{C})\big)$ by the conjugate action. It is known that \eqref{18110307} is an algebraic curve parametrized by $M,L$, which are the derivatives of the holonomies of the meridian-longitude pair of the cusp of $\mathcal{M}$. If \eqref{18110307} is given by $f(M,L)=0$, then $M=L=1$ corresponds to the complete hyperbolic structure of $\mathcal{M}$ and a point of
\begin{equation*}
(f(M, L)=0)\cap (M^pL^q=1)
\end{equation*}
gives rise to the hyperbolic structure of $\mathcal{M}_{p/q}$. Moreover, by letting 
\begin{equation*}
u=\log\; M, \quad v=\log\; L,
\end{equation*} 
we have the following theorem \cite{NZ}:
\begin{theorem}[Neumann-Zagier]
Locally near $(u,v)=(0,0)$, $v$ is represented as an analytic function of $u$ in the following form:
\begin{equation}\label{18110802}
v=\sum_{i:\;odd} c_iu^i.
\end{equation} 
Moreover each $c_i$ is algebraic and, in particular, $c_1$ is non-real.  
\end{theorem}
The function \eqref{18110802} is called the \textit{\textbf{Neumann-Zagier potential function}} of $\mathcal{M}$ and $c_1$ is called the \textit{\textbf{cusp shape}} of $\mathcal{M}$.
\begin{rmk}\label{18112604}
{\normalfont Note that \eqref{18110802} depends on the choice of meridian-longitude pair of the cusp. For instance, if 
\begin{equation*}
m'=am+bl, \quad l'=cm+dl  \quad (ad-bc=1)
\end{equation*}
is another basis, then the cusp shape $c'_1$ under this new basis is 
\begin{equation*}
c'_1=\dfrac{ac_1+b}{cc_1+d}.
\end{equation*} 
 } 
\end{rmk}

Now we state Theorem \ref{18110801} more precisely as follows \cite{NZ}:
\begin{theorem}[Neumann-Zagier]\label{NZ1}
Let $\mathcal{M}$ be a $1$-cusped hyperbolic $3$-manifold whose Neumann-Zagier potential function is given as
\begin{equation*}
v=c_1u+c_3u^3+\cdots. 
\end{equation*} 
Then  
\begin{equation}\label{18121106}
\text{\normalfont\ vol}(\mathcal{M}_{p/q})=\text{\normalfont\ vol}(\mathcal{M})+\Theta_{\mathcal{M}}(p,q)
\end{equation}
where 
\begin{equation}\label{18110808}
\begin{gathered}
\Theta_{\mathcal{M}}(p,q)=-\text{\normalfont\ Im}\; c_1\dfrac{\pi^2}{|A|^2}+\dfrac{1}{4}\text{\normalfont\ Im}\Big[\dfrac{1}{2}c_3\Big(\dfrac{2\pi}{A}\Big)^4+\dfrac{1}{3}\Big(2c_3^2\dfrac{q}{A}-c_5\Big)\Big(\dfrac{2\pi}{A}\Big)^6\\
+\dfrac{1}{4}\Big(12c_3^3\dfrac{q^2}{A^2}-8c_3c_5\dfrac{q}{A}+c_7\Big)\Big(\dfrac{2\pi}{A}\Big)^8\Big)+\cdots\Big]
\end{gathered}
\end{equation}
and $A=p+c_1q$. In particular, $\Theta_{\mathcal{M}}$ converges for sufficiently large $|p|+|q|$. 
\end{theorem}

As noted in Theorem \ref{18110801}, the following leading term 
\begin{equation}\label{18112603}
-\text{\normalfont\ Im}\; c_1\dfrac{\pi^2}{|A|^2}=-\dfrac{\pi^2\text{\normalfont\ Im}\; c_1}{|p+c_1q|^2}
\end{equation}
dominates the behavior of $\Theta_{\mathcal{M}}$ and so \eqref{18111601} can be rewritten as 
\begin{equation}
\text{\normalfont\ vol}(\mathcal{M}_{p/q})=\text{\normalfont\ vol}(\mathcal{M})-\dfrac{\pi^2\text{\normalfont\ Im}\; c_1}{|p+c_1q|^2}+O\Big(\dfrac{1}{|p+c_1q|^4}\Big).
\end{equation}

We call $\Theta_{\mathcal{M}}$ the \textit{\textbf{volume function}} of $\mathcal{M}$, $|p+c_1q|^2$ the \textit{\textbf{quadratic form associate to}} $\mathcal{M}$ and denote $|p+c_1q|^2$ by $Q_{\mathcal{M}}(p,q)$.

The following are examples of volume functions for various well-known hyperbolic $3$-manifolds \cite{HM}.

\begin{example}
If $\mathcal{M}$ is the figure eight knot complement, under the standard meridian-longitude pair of the cusp, its volume formula is 
\begin{equation*}
\text{vol}\big(\mathcal{M}(p/q)\big)=\text{vol}\big(\mathcal{M}\big)-\dfrac{2\sqrt{3}\pi^2}{p^2+12q^2}+\dfrac{4\sqrt{3}(p^4-72p^2q^2+144q^4)\pi^4}{3(p^2+12q^2)^4}+\cdots.
\end{equation*}
\end{example}

\begin{example}
If $\mathcal{M}$ is the sister of the figure eight knot complement, then there exists a basis of the cusp such that the volume formula of $\mathcal{M}$ under this is 
\small\begin{equation*}
\begin{gathered}
\text{vol}\big(\mathcal{M}(p/q)\big)=\text{vol}\big(\mathcal{M}\big)-\dfrac{\sqrt{3}\pi^2}{2(p^2+pq+q^2)}
+\dfrac{\pi^4\big((2p+q)^4-18q^2(2p+q)^2+9q^4\big)}{64\sqrt{3}(p^2+pq+q^2)^4}+\cdots.
\end{gathered}
\end{equation*}
\end{example}

\begin{example}
Let $\mathcal{M}$ be the Whitehead link complement and $\mathcal{M}(\infty, p/q)$ be its $(\infty, p/q)$-Dehn filling of it. Then its volume formula under its geometric bases is
\begin{equation*}
\text{vol}\big(\mathcal{M}(\infty, p/q)\big)=\text{vol}\big(\mathcal{M}\big)-\dfrac{2\pi^2}{p^2+4q^2}+\dfrac{\pi^4(p^4-24p^2q^2+16q^4)}{3(p^2+4q^2)^4}+\cdots.
\end{equation*}
\end{example}

\begin{example}
Let $\mathcal{M}$ be the Whitehead link sister complement and $\mathcal{M}(\infty, p/q)$ be its $(\infty, p/q)$-Dehn filling. Then its volume formula under its geometric bases is
\begin{equation*}
\text{vol}\big(\mathcal{M}(\infty, p/q)\big)=\text{vol}\big(\mathcal{M}\big)-\dfrac{\pi^2}{p^2+q^2}+\dfrac{\pi^4(p^4-12p^3q-6p^2q^2+12pq^3+q^4)}{24(p^2+q^2)^4}+\cdots.
\end{equation*}
\end{example}

If $c_1$ is quadratic, then \eqref{18112603} is of the form
\begin{equation*}
\dfrac{\pi^2\sqrt{d}}{ap^2+bpq+cq^2}
\end{equation*}
for some $a,b,c,d\in \mathbb{Q}$. Throughout the paper, for simplicity's sake, we assume the cusp shape of a given $1$-cusped hyperbolic $3$-manifold is of the form $\sqrt{-d}$ ($d\in \mathbb{N}$) and so the leading term of its volume function is  
\begin{equation}\label{18112605}
-\dfrac{\pi^2\sqrt{d}}{p^2+dq^2}.
\end{equation}
The general case can be handled with a slight modification of this simple case. 

Presumably, the volume function of any 1-cusped hyperbolic 3-manifold is transcendental, but, as far as we are aware, nothing has been known. In the following, we show that the Neumann-Zagier function is not linear. This observation will play a key role in the proof of the main theorem.

\begin{lemma}\label{NZ}
Let $\mathcal{M}$ be a $1$-cusped hyperbolic $3$-manifold with its Neumann-Zagier function 
\begin{equation}\label{18110805}
v=c_1 u+c_{3}u^{3}+\cdots.
\end{equation}
Then there exists $i$ ($i\geq 2$) such that $c_{2i-1}\neq 0$. In other words, 
\begin{equation*}
v\neq c_1 u.
\end{equation*}
\end{lemma}

We use the following well-known theorem in number theory to prove the lemma \cite{B}. 
\begin{theorem}[Gelfond-Schneider]\label{18121104}
Let $a,b\in \overline{\mathbb{Q}}$ such that $a,b\neq 0$. Then $\log b/\log a$ is either rational or transcendental. 
\end{theorem}

\begin{proof}[Proof of Lemma \ref{NZ}]
Suppose its Neumann-Zagier potential function is of the following form
\begin{equation*}
v=c_1 u.
\end{equation*}
By Thurston's Dehn filling theorem, there exists $u$ such that both $e^u$ and $e^v=e^{c_1 u} $ are algebraic. By Theorem \ref{18121104}, 
\begin{equation*}
\dfrac{\log\;e^{c_1 u}}{\log\; e^{u}}=\dfrac{c_1 u}{u }=c_1
\end{equation*}
is either rational or transcendental. But this contradicts the fact that the cusp shape $c_1$ is a non-real algebraic number (Theorem \ref{NZ1}). 
\end{proof}

By Lemma \ref{NZ} and Theorem \ref{NZ1}, the following corollary is immediate: 

\begin{corollary}
For a given $1$-cusped hyperbolic $3$-manifold $\mathcal{M}$, 
\begin{equation*}
\Theta_{\mathcal{M}}(p,q) \neq -\dfrac{\pi^2\text{\normalfont\ Im}\; c_1}{|p+c_1q|^2}.
\end{equation*}
\end{corollary}

\subsection{o-minimal structures}\label{o-minimal}
In this section we define an o-minimal structure and study its basic properties. Let us first start with the following definition:
\begin{definition}
An \textbf{algebraic} subset of $\mathbb{R}^n$ is a subset of the form
\begin{equation*}
X=\bigcup^{p}_{i=1}\bigcap^{q}_{j=1}X_{ij}
\end{equation*}
where each $X_{ij}$ is $\{f_{ij}=0\}$ for some $f_{ij}\in \mathbb{R}[x_1,\dots, x_n]$.
\end{definition}
Unlike complex algebraic subsets, real algebraic subsets are typically not closed under projections, so we generalize the above definition as follows: 
\begin{definition}\label{18123101}
A \textbf{semialgebraic} subset of $\mathbb{R}^n$ is a subset of the form
\begin{equation*}
X=\bigcup^{p}_{i=1}\bigcap^{q}_{j=1}X_{ij}
\end{equation*}
where each $X_{ij}$ is either $\{f_{ij}=0\}$ or $\{f_{ij}>0\}$, $f_{ij}\in \mathbb{R}[x_1,\dots, x_n]$.
\end{definition}
The following theorem is due to Tarski-Seidenberg \cite{Dries}:
\begin{theorem}[Tarski-Seidenberg]
Let $A$ be a semialgebraic set in $\mathbb{R}^{n+1}$ and 
\begin{equation*}
\pi \;:\; \mathbb{R}^{n+1}\longrightarrow\mathbb{R}^n,
\end{equation*} 
be the projection on the first $n$-coordinates. Then $\pi(A)$ is semialgebraic.
\end{theorem}

The following is an analytic analogue of Definition \ref{18123101}. 

\begin{definition}
Let $U$ be an open subset of $\mathbb{R}^n$ and $\mathcal{O}(U)$ be the ring of real analytic functions on $U$. Let $S(\mathcal{O}(U))$ be the collection of all the sets of the form 
\begin{equation*}
\bigcup^{p}_{i=1}\bigcap^{q}_{j=1}X_{ij}
\end{equation*}
where each $X_{ij}$ is either $\{f_{ij}=0\}$ or $\{f_{ij}>0\}$, $f_{ij}\in \mathcal{O}(U)$. We call $X\subset \mathbb{R}^n$ \textbf{semianalytic} if each $a\in \mathbb{R}^n$ has a neighbourhood $U$ such that $X\cap U\in S(\mathcal{O}(U))$. 
\end{definition}
A projection of a semianalytic subset is not semianalytic in general, so we further refine the definition as follows:
\begin{definition}
A subset $X$ of $\mathbb{R}^n$ is \textbf{subanalytic} if, locally, $X$ is a projection of a relatively compact seminalytic set. 
\end{definition} 

The above term was first defined by Hironaka \cite{H1}. He also proved the following uniformization theorem \cite{H2}:
\begin{theorem}
Let $X$ be a closed subanalytic subset of $\mathbb{R}^n$. Then $X$ is the image of a proper real analytic mapping $\phi\;:\;Y \longrightarrow \mathbb{R}^n$, where $Y$ is a real analytic subset whose dimension is equal to dim $X$. 
\end{theorem}

The notion of o-minimal structure is an axiomatic approach to semialgebraic and subanalytic geometries. A benefit of this approach is one can study these two (and even more) different categories from a single unified perspective. We will not provide all details but some essential concepts and theorems that we need. Please see \cite{Van der} and references therein for further study.

\begin{definition}
A \textbf{prestructure} is a sequence $\chi=(\chi_n\;|\;n\geq 1)$, where each $\chi_n$ is a collection of subsets of $\mathbb{R}^n$. A restructure $\chi$ is called a \textbf{structure (over the real field)} if, for all $n, m\geq 1$, the following conditions are satisfied.

(1)\quad $\chi_n$ is a boolean algebra (under the usual set-theoretic operations).

(2)\quad $\chi_n$ contains every semialgebraic subset of $\mathbb{R}^n$.

(3)\quad If $A\in \chi_n$ and $B\in \chi_m$, then $A\times B\in \chi_{n+m}$.

(4)\quad If $m\geq n$ and $A\in \chi_m$, then $\pi[A]\in \chi_n$, where $\pi\;:\;\mathbb{R}^m\longrightarrow \mathbb{R}^n$ is projection onto the first $n$ coordinates. \\
If $\chi$ is a structure and $X\subset \mathbb{R}^n$, we say that $X$ is \textbf{definable} in $\chi$ if $X\in \chi_n$. If $\chi$ is a structure and, in addition, 

(5)\quad the boundary of every set in $\chi_1$ is finite, \\
then $\chi$ is called an \textbf{o-minimal} structure (over the real field).  
\end{definition}

\begin{example}
The simplest example of o-minimal structure is the class of semialgebraic sets. 
\end{example}

\begin{example}
The class of subanalytic sets is not an o-minimal structure and thus we modify the definition slightly as follows. A \textit{\textbf{globally subanalytic}} set is a set which is subanalytic as a subset of the compact real analytic manifold $\mathbb{P}^n(\mathbb{R})$. That is, it is a set which is subanalytic in any standard euclidean chart $\mathbb{R}^n$ of  $\mathbb{P}^n(\mathbb{R})$. An arbitrarily compact subanalytic set of $\mathbb{R}^n$ is a typical example of a globally subanalytic set. By the works of Gabrielov and van der Dries, the class of globally subanalytic sets is an o-minimal structure \cite{Dries}\cite{G}. This o-minimal structure is denoted by $\mathbb{R}_{an}$. 
\end{example}

The o-minimal structure that we are mainly interested in throughout this paper will be $\mathbb{R}_{an}$, thanks to the following lemma:

\begin{lemma}\label{18101002}
Let $\mathcal{M}$ and $\Theta_{\mathcal{M}}(p,q)$ be the same as above. We extend the domain of $\Theta_{\mathcal{M}}$ and consider it as a real analytic function of two variables $x,y$ defined over sufficiently large $|x|+|y|$. Then 
$\Theta_{\mathcal{M}}(x,y)$ is a globally subanalytic function. That is, $\Theta_{\mathcal{M}}(x,y)\in \mathbb{R}_{an}$. 
\end{lemma}
\begin{proof}
It is enough to prove that the function is analytic at infinity. In other words, we show that the transformations of $\Theta_{\mathcal{M}}(x,y)$ under   
\begin{equation}
x'=\frac{x}{y}, \quad y'=\frac{1}{y}.
\end{equation}
and
\begin{equation}
x'=\frac{1}{x}, \quad y'=\frac{y}{x}
\end{equation}
are analytic near $(0,0)$. We will only consider the first case since the second case can be treated similarly.  

By the formula given in \eqref{18110808},  up to a scalar multiple, an arbitrarily term of $\Theta(x,y)$ is either one of the following forms:   
\begin{equation*}
\text{Im} \dfrac{y^j}{A^{2n+j}}=\text{Im} \dfrac{y^j\overline{A}^{2n+j}}{A^{2n+j}\overline{A}^{2n+j}}=\dfrac{\text{Im}\;y^j\overline{A}^{2n+j}}{|A|^{4n+2j}}
\end{equation*}
or 
\begin{equation}\label{18111103}
\text{Re} \dfrac{y^j}{A^{2n+j}}=\text{Re} \dfrac{y^j\overline{A}^{2n+j}}{A^{2n+j}\overline{A}^{2n+j}}=\dfrac{\text{Re}\;y^j\overline{A}^{2n+j}}{|A|^{4n+2j}}.
\end{equation}

To simplify notation, we assume $c_1=i$ and so $A=x+iy$. Then 
\begin{equation*}
\dfrac{\text{Im}\;y^j\overline{A}^{2n+j}}{|A|^{4n+2j}}=\dfrac{\text{Im}\;y^j(x-iy)^{2n+j}}{(x^2+y^2)^{2n+j}},
\end{equation*}
and, by substituting
\begin{equation*}
x=\dfrac{x'}{y'}, \quad y=\dfrac{1}{y'}
\end{equation*}
into it, we get
\begin{equation}\label{18111102}
\begin{gathered}
\dfrac{\text{Im}\;\Big(\big(\frac{1}{y'}\big)^j\big((\frac{x'}{y'})-i(\frac{1}{y'})\big)^{2n+j}\Big)}{\big((\frac{x'}{y'})^2+(\frac{1}{y'})^2\big)^{2n+j}}
=\dfrac{\text{Im}\;\Big(\big(\frac{1}{y'}\big)^j\big(\frac{x'-i}{y'}\big)^{2n+j}\Big)}
{\Big(\frac{(x')^2+1}{(y')^2}\Big)^{2n+j}}
=\dfrac{\text{Im}\;\Big(\frac{(x'-i)^{2n+j}}{(y')^{2n+2j}}\Big)}{\Big(\frac{(x')^2+1}{(y')^2}\Big)^{2n+j}}\\
=\dfrac{\text{Im}\;\big((y')^{2n}(x'-i)^{2n+j}\big)}{\big((x')^2+1\big)^{2n+j}}.
\end{gathered}
\end{equation}
Since 
\begin{equation*}
\dfrac{1}{(x')^2+1}
\end{equation*}
is analytic for $x'\in \mathbb{R}$, \eqref{18111102} is also analytic over $\mathbb{R}^2(:=(x',y'))$. Similarly one can show that \eqref{18111103} is also analytic over the same domain.

Since $\Theta_{\mathcal{M}}(\frac{x'}{y'}, \frac{1}{y'})$ converges near $(0,0)$ and each term of it is analytic over $\mathbb{R}^2$, it is analytic near $(0,0)$. In conclusion, $\Theta_{\mathcal{M}}(x,y)$ is analytic at infinity and so $\Theta_{\mathcal{M}}(x,y)\in \mathbb{R}_{an}$.
\end{proof}

We introduce two more well-known o-minimal structures. 
\begin{example}
Let $\chi_n$ be the collection of subsets of $\mathbb{R}^n$ of the form $X=\pi(f^{-1}(0))$ where 
\begin{equation*}
f(x_1,\dots, x_m)=Q(x_1,\dots,x_m, e^{x_1}, \dots, e^{x_m})
\end{equation*} 
for some polynomial $Q\in \mathbb{R}[X_1,\dots, X_{2m}]$ and 
\begin{equation*}
\pi:\;\mathbb{R}^m\longrightarrow \mathbb{R}^n
\end{equation*} 
is the projection on the first $n$ coordinates. It was proved by A. J. Wilkie in \cite{W1} that this collection satisfies the axioms of the o-minimal structure. We denote this o-minimal structure by $\mathbb{R}_{exp}$.  
\end{example}
\begin{example}
The last example is $\mathbb{R}_{an, exp}$, which is generated by the union of $\mathbb{R}_{an}$ and $\mathbb{R}_{exp}$. The o-minimality of this set is due to van der Dries and Miller \cite{DM}. 
\end{example}

A nice thing about o-minimal structures is that there are no pathological things (such as Cantor sets, Borel sets, nonmeasurable sets, etc) in such structures, and so one can develop a tame topology on them. For instance, it is known that every definable set has only finitely connected components, and the boundary and interior of a definable set are again definable. The following theorems, which are key tools in the proof of our main theorem, exhibit other fine features of o-minimal structures.


\begin{theorem}\label{18103002}
Let $S\subset \mathbb{R}^m\times \mathbb{R}^n$ be definable, and for each $a\in \mathbb{R}^m$, put
\begin{equation}\label{18111701}
S_a:=\{b\in \mathbb{R}^n\;:\;(a,b)\in S\}.
\end{equation} 
Then there is a number $M=M(S)\in \mathbb{N}$ such that for each $a\in \mathbb{R}^m$ the fiber $S_a$ has at most $M$ definably connected components.\footnote{Note that a set is definably connected if and only if it is connected in the usual sense.} In particular, if $S_a$ is a finite set, then the cardinality of $S_a$ is at most $M$. 
\end{theorem}
\begin{proof}
See Corollaries 3.6 and 3.7 of Chapter 3 in \cite{Van der}. 
\end{proof}

\begin{theorem}\label{18103004}
Let $S\subset \mathbb{R}^m\times \mathbb{R}^n$ be definable. For $d\in\{-\infty, 0, 1, \dots, n\}$ put 
\begin{equation*}
S(d):=\{a\in \mathbb{R}^n\;:\; \text{\normalfont\ dim}\;S_a=d\}.
\end{equation*}
Then $S(d)$ is definable and the part of $S$ above $S(d)$ has dimension given by
\begin{equation*}
\text{\normalfont\ dim}\; \big(\bigcup_{a\in S(d)}\{a\}\times S_a \big)=\text{\normalfont\ dim}\;(S(d))+d.
\end{equation*}
\end{theorem}
\begin{proof}
See Proposition 1.5 of Chapter 4 in \cite{Van der}.
\end{proof}

\section{Preliminary Lemmas}\label{Preliminary}
In this section, we prove several preliminary lemmas needed to prove the main result. Since the proof of the main result only relies on the statements of these lemmas, a trusting reader can skip ahead details of the proofs at first reading.  
\subsection{Two Dehn fillings of the same volume} The first lemma roughly says that if two different Dehn fillings have the same volume, then the gap between their associated quadratic forms is not that big but uniformly bounded. 

\begin{lemma}\label{18110304}
Let $\mathcal{M}$ be a $1$-cusped hyperbolic $3$-manifold and $Q_{\mathcal{M}}$ be the quadratic form associated with it. Then there exists $m=m(\mathcal{M})$ such that if
\begin{equation}\label{18112501}
\text{\normalfont\ vol}(\mathcal{M}_{p/q})=\text{\normalfont\ vol}(\mathcal{M}_{p'/q'}),
\end{equation}
then
\begin{equation*}
|Q_{\mathcal{M}}(p,q)-Q_{\mathcal{M}}(p',q')|<m.
\end{equation*}
\end{lemma}
\begin{proof}
We denote 
\begin{equation*}
\Theta_{\mathcal{M}}(p,q)-\dfrac{\pi^2\text{Im}\;c_1}{Q_{\mathcal{M}}(p,q)}
\end{equation*}
by $E(p,q)$ and rewrite the Neumann-Zagier volume formula \eqref{18121106} as
\begin{equation*}
\text{vol}(\mathcal{M}_{p/q})=\text{vol}(\mathcal{M})-\dfrac{\pi^2\text{Im}\;c_1}{Q_{\mathcal{M}}(p,q)}+E(p,q).
\end{equation*}
If 
\begin{equation*}
\text{vol}(\mathcal{M}_{p/q})=\text{vol}(\mathcal{M}_{p'/q'}),
\end{equation*}
then clearly
\begin{equation}\label{181023011}
-\dfrac{\pi^2\text{Im}\;c_1}{Q_{\mathcal{M}}(p,q)}+E(p,q)=-\dfrac{\pi^2\text{Im}\;c_1}{Q_{\mathcal{M}}(p',q')}+E(p',q').
\end{equation}

To simplify notation, let 
\begin{equation*}
r=Q_{\mathcal{M}}(p,q), \quad r'=Q_{\mathcal{M}}(p',q').
\end{equation*}
Since
\begin{equation*}
\begin{gathered}
E(p,q)=O\Big(\dfrac{1}{Q_{\mathcal{M}}(p,q)^2}\Big)=O\Big(\dfrac{1}{r^2}\Big), \\
E(p',q')=O\Big(\dfrac{1}{Q_{\mathcal{M}}(p',q')^2}\Big)=O\Big(\dfrac{1}{(r')^2}\Big),
\end{gathered}
\end{equation*}
combining with \eqref{181023011}, we find $C=C(\mathcal{M})>0$ such that 
\begin{equation}\label{18112801}
\Big|\dfrac{1}{r}-\dfrac{1}{r'}\Big|<C\Big(\dfrac{1}{r^2}+\dfrac{1}{(r')^2}\Big).
\end{equation}
Now we claim the following:
\begin{claim}\label{18120401}
Let $C$ be the same as above. If $r, r'$ are sufficiently large satisfying $|r-r'|>2C$, then  
\begin{equation}\label{18102302}
\Big|\dfrac{1}{r}-\dfrac{1}{r'}\Big|>C\Big(\dfrac{1}{r^2}+\dfrac{1}{(r')^2}\Big).
\end{equation}
\end{claim}

Assuming the claim, since \eqref{18112801} and \eqref{18102302} are opposite to each other, we get 
\begin{equation*}
|r-r'|=|Q_{\mathcal{M}}(p,q)-Q_{\mathcal{M}}(p',q')|\leq 2C 
\end{equation*}
for $|p|+|q|, |p'|+|q'|$ sufficiently large, and thus obtain the desired conclusion. 
\textit{Proof of Claim \ref{18120401}}
Without loss of generality, we assume $\dfrac{1}{r}>\dfrac{1}{r'}$. Then \eqref{18102302} is equivalent to 
\begin{equation}\label{18120406}
\dfrac{r'-r}{rr'}>C\Big(\dfrac{r^2+(r')^2}{r^2(r')^2}\Big),
\end{equation}
which is simplified as
\begin{equation*}
r'-r>C\Big(\dfrac{r^2+(r')^2}{rr'}\Big)
\end{equation*}
and further as 
\begin{equation}\label{18102404}
r(r')^2-r^2r'>C(r^2+(r')^2).
\end{equation}
Now \eqref{18102404} is equivalent to 
\begin{equation*}
(r')^2(r-C)>r^2(r'+C),
\end{equation*}
and  
\begin{equation}\label{18121001}
\dfrac{(r')^2}{r^2}>\dfrac{r'+C}{r-C}.
\end{equation}
If we put $r'=r+k$, then \eqref{18121001} implies
\begin{gather} 
\dfrac{(r+k)^2}{r^2}>\dfrac{r+k+C}{r-C}\nonumber\\
\Longrightarrow\dfrac{r^2+2rk+k^2}{r^2}>\dfrac{r-C+(k+2C)}{r-C}\nonumber\\
\Longrightarrow\dfrac{2rk+k^2}{r^2}>\dfrac{k+2C}{r-C}\nonumber\\
\Longrightarrow(2rk+k^2)(r-C)>r^2(k+2C)\nonumber\\
\Longrightarrow 2r^2k+k^2r-2rkC-k^2C>r^2k+2r^2C\nonumber\\
\Longrightarrow r^2k+k^2r-2rkC-2r^2C>+k^2C\nonumber\\
\Longrightarrow r(rk+k^2-2rC-2kC)>k^2C\nonumber\\
\Longrightarrow r(r+k)(k-2C)>k^2C.\label{19032807}
\end{gather}
Clearly \eqref{19032807} holds for $k>2C$ and sufficiently large $r$. This completes the proof of the claim and so the proof of Lemma \ref{18110304}. 
\end{proof}


\subsection{Proof of Lemma \ref{18112502}}
In this section, we prove Lemma \ref{18112502}. Let us first remind the following well-known theorem of Weierstrass \cite{KP}:
\begin{theorem}[Weierstrass division theorem]\label{Weier}
Let 
\begin{equation*}
g(\bold{x},y)\sum_{\alpha\in (\mathbb{Z}^+)^n}\sum^{\infty}_{j=0}g_{\alpha, j}\bold{x}^{\alpha}y^j,
\end{equation*}
$\bold{x}=(x_1, \dots, x_n)\in \mathbb{R}^n, y\in \mathbb{R}$, be real analytic in a neighborhood of $(\bold{0},0)\in \mathbb{R}^n\times \mathbb{R}$ and suppose there exists $k\in \mathbb{N}$ such that 
\begin{equation*}
g_{0,0}=g_{0,1}=\dots=g_{0,k-1}=0, \quad \quad g_{0,k}=1.
\end{equation*}
If $f(\bold{x},y)$ is real analytic in a neighborhood of $(\bold{0},0)\in \mathbb{R}^n\times \mathbb{R}$, then there exist unique real analytic functions $h(\bold{x},y)$ and $r(\bold{x},y)$ such that 
\begin{equation*}
f(\bold{x},y)= g(\bold{x},y)h(\bold{x},y)+r(\bold{x},y)
\end{equation*} 
where the degree of $y$ in any term of $r(\bold{x},y)$ is strictly less than $k$.
\end{theorem}

Using the Weierstrass division theorem, we prove 
\begin{lemma}\label{18112601}
Let $\mathcal{M}$ be a $1$-cusped hyperbolic $3$-manifold whose cusp shape is quadratic, $\Theta_{\mathcal{M}}(x,y)$ and $Q_{\mathcal{M}}(x,y)$ be its volume function and associated quadratic form respectively. Suppose 
\begin{equation}\label{18101401}
\Theta_{\mathcal{M}}(x,y)=\Theta_{\mathcal{M}}(z,w)
\end{equation}
and
\begin{equation}\label{18102401}
\dfrac{1}{\mathcal{Q}_{\mathcal{M}}(x,y)}=\dfrac{1}{\mathcal{Q}_{\mathcal{M}}(z,w)}
\end{equation}
are equivalent for $|x|+|y|, |z|+|w|$ sufficiently large. Then there exists an analytic function $h$ such that 
\begin{equation}\label{18102407}
\Theta_{\mathcal{M}}(x, y) = h\Big(\dfrac{1}{\mathcal{Q}_{\mathcal{M}}(x,y)}\Big). 
\end{equation}
\end{lemma}
\begin{proof}
As defined in the proof of Lemma \ref{18101002}, let 
\begin{equation*}
\Psi(x,y):=\Theta_{\mathcal{M}}\Big(\dfrac{x}{y}, \dfrac{1}{y}\Big)
\end{equation*}
and consider it as an analytic functions over a compact domain $D(\subset \mathbb{R}^2)$ containing the origin. Similarly put
\begin{equation*}
g(x,y):=Q_{\mathcal{M}}\Big(\dfrac{x}{y}, \dfrac{1}{y}\Big)=\dfrac{y^2}{x^2+d}=\dfrac{y^2}{d}\Big(1-\frac{x^2}{d}+\frac{x^4}{d^2}-\cdots\Big)
\end{equation*}
and consider $g(x,y)$ as an analytic functions over $D$.

If \eqref{18101401} and \eqref{18102401} define the same variety, then, equivalently, it implies   
\begin{equation*}
\Psi(x,y)=\Psi(z,w)
\end{equation*}
and 
\begin{equation*}
g(x,y)=g(z,y)
\end{equation*}
define the same variety over $D\times D$. By the Weierstrass division theorem, there exist $\Psi_1(x,y)$ and $r(x,y)$ such that  
\begin{equation}\label{18102402}
\Psi(x,y)=g(x,y)\Psi_1(x,y)+r(x,y)
\end{equation}
where $r(x,y)$ is of the following form
\begin{equation}\label{18101404}
r_1(x)y+r_0(x).
\end{equation}

Now we claim 
\begin{claim}
In \eqref{18101404}, $r_1(x)=r_0(x)=0$. That is, $r(x,y)=0$ in \eqref{18102402} and so 
\begin{equation*}
g(x,y)\;|\;\Psi(x,y). 
\end{equation*} 
\end{claim}
\begin{proof}
Since 
\begin{equation*}
g(x,0)=\Psi(x,0)=0
\end{equation*} 
for $0\leq x<\epsilon$ ($\epsilon$ is sufficiently small), we have
\begin{equation*}
r(x,0)=r_0(x)=0.
\end{equation*}

Now let $\Psi_1(x,y)$ be of the following form:
\begin{equation*}
q_0(x)+q_1(x)y+\cdots.
\end{equation*}
Since \eqref{18101401} and \eqref{18102401} define the same variety, for any $c>0$, there exists some $C>0$ (depending on $c$) such that 
\begin{equation*}
g(x,y)=c \Longleftrightarrow \Psi(x,y)=C.
\end{equation*}
We parametrize two curves as 
\begin{equation*}
g\big(x_t,y_t\big)=c \quad \text{and}\quad  \Psi\big(x_t,y_t\big)=C.
\end{equation*}
By \eqref{18102402},  
\begin{equation*}
C=\Psi(x_t,y_t)=g(x_t,y_t)\Psi_1(x_t,y_t)+r(x_t,y_t)=c\Psi_1(x_t,y_t)+r_1(x_t)y_t,
\end{equation*}
which is   
\begin{equation}\label{18101406}
C=c(q_0(x_t)+q_1(x_t)y_t+\cdots)+r_1(x_t)y_t.
\end{equation}
Since the degree of $y$ in $g(x,y)$ is even, we have
\begin{equation}\label{18113001}
g(x_t,-y_t)=c
\end{equation}
and so 
\begin{equation}\label{18120409}
\Psi(x_t, -y_t)=C
\end{equation}
Applying \eqref{18113001} and\eqref{18120409} to \eqref{18102402}, we get
\begin{equation}
C=\Psi(x_t,-y_t)=g(x_t,-y_t)\Psi_1(x_t,-y_t)+r(x_t,-y_t)=c\Psi_1(x_t,-y_t)-r_1(x_t)y_t,
\end{equation}
which is 
\begin{equation}\label{18101407}
C=c\big(q_0(x_t)-q_1(x_t)y_t+\cdots\big)-r_1(x_t)y_t.
\end{equation}
Comparing \eqref{18101406} and \eqref{18101407}, we have
\begin{equation*}
c\big(q_0(x_t)+q_1(x_t)y_t+\cdots\big)+r_1(x_t)y_t=c\big(q_0(x_t)-q_1(x_t)y_t+\cdots\big)-r_1(x_t)y_t,
\end{equation*}
which implies
\begin{equation}\label{18102405}
\big(cq_1(x_t)+r_1(x_t)\big)y_t+\cdots=-\big(cq_1(x_t)+r_1(x_t)\big)y_t+\cdots
\end{equation}
and 
\begin{equation}\label{18121101}
r_1(x_t)y_t+\sum_{i=1}^{\infty} cq_{2i-1}(x_t)y_t^{2i-1}=0.
\end{equation}
Since  
\begin{equation*}
g(x_t,y_t)=c
\end{equation*}
and $c$ is arbitrarily, we conclude  
\begin{equation}\label{18121102}
r_1(x)y+g(x,y)\sum_{i=1}^{\infty} q_{2i-1}(x)y^{2i-1}=0
\end{equation}
for any $x,y$. Recall the degree of $y$ in each term of $g(x,y)$ is at least $2$, and so $r_1(x)$ is the coefficient of $y$ in \eqref{18121102}. This implies $r_1(x)=0$ and thus completes the proof of the claim. 
\end{proof}
By the claim, 
\begin{equation}\label{18111401}
\Psi(x,y)=\Psi(z,w)
\end{equation}
is rewritten as 
\begin{equation}\label{18111401}
g(x,y)\Psi_1(x,y)=g(z,w)\Psi_1(z,w).
\end{equation}
By the assumption, \eqref{18111401} and  
\begin{equation}\label{18113002}
g(x,y)=g(z,w)
\end{equation}
are equivalent to each other, and so we get the equivalence between 
\begin{equation*}
\Psi_1(x,y)=\Psi_1(z,w)
\end{equation*} 
and \eqref{18113002}.

Let 
\begin{equation*}
\Psi_1(0,0)=c_1,
\end{equation*}
then, by applying the claim again, we have
\begin{equation*}
g(x,y)\;|\;\big(\Psi_1(x,y)-c_1\big).
\end{equation*}  
Now we construct $\Psi_i, c_i$  ($i\geq 2$) inductively and conclude 
\begin{equation*}
\Psi(x,y)=h(g(x,y))
\end{equation*}
where $h(t)=c_1t+c_2t^2+\cdots$, implying 
\begin{equation*}
\Theta_{\mathcal{M}}(x, y) = h\Big(\dfrac{1}{Q_{\mathcal{M}}(x,y)}\Big).
\end{equation*}
This completes the proof of the lemma. 
\end{proof}

Using the above lemma, we now prove Lemma \ref{18112502}. 

\begingroup
\def\thetheorem{\ref{18112502}}
\begin{lemma}
Let $\mathcal{M}$ be a $1$-cusped hyperbolic $3$-manifold whose cusp shape is quadratic. Let $\Theta_{\mathcal{M}}(x, y)$ be its volume function and $Q_{\mathcal{M}}(x,y)$ be the quadratic form associated with it. Then there is no analytic function $h(t)$ such that 
\begin{equation}\label{18102406}
\Theta_{\mathcal{M}}(x,y)=h\Big(\dfrac{1}{Q_{\mathcal{M}}(x,y)}\Big).
\end{equation}
\end{lemma}
\addtocounter{theorem}{-1}
\endgroup

\begin{proof}
By Theorem \ref{NZ1}, we have 
\begin{equation}\label{18101409}
\begin{gathered}
\Theta_{\mathcal{M}}(x,y)=-\text{Im}\;c_1\dfrac{\pi^2}{|A|^2}+\dfrac{1}{4}\text{Im}\Big[\dfrac{1}{2}c_3\Big(\dfrac{2\pi}{A}\Big)^4+\dfrac{1}{3}\Big(2c_3^2\dfrac{q}{A}-c_5\Big)\Big(\dfrac{2\pi}{A}\Big)^6\\
+\dfrac{1}{4}\Big(12c_3^3\dfrac{q^2}{A^2}-8c_3c_5\dfrac{q}{A}+c_7\Big)\Big(\dfrac{2\pi}{A}\Big)^8+\cdots \Big]
\end{gathered}
\end{equation}
where $A=\dfrac{1}{x+c_1y}$. By Theorem \ref{NZ}, there exists the smallest $2i-1$ (where $i\geq 2$) such that $c_{2i-1}\neq 0$.  As discussed in Section \ref{NZAVF}, we assume $c_1$ is of the form $\sqrt{-d}$ for some $d\in \mathbb{N}$ (and so $Q_{\mathcal{M}}(x,y)=x^2+dy^2$). 

Suppose there exists an analytic function 
\begin{equation*}
h(t)=a_1t+a_{i}t^{i}+\cdots
\end{equation*}
such that $\Theta_{\mathcal{M}}(x,y)$ in \eqref{18101409} is equal to 
\begin{equation}\label{18111603}
\begin{gathered}
h\Big(\dfrac{1}{Q_\mathcal{M}(x,y)}\Big)=\dfrac{a_1}{x^2+dy^2}+\dfrac{a_i}{(x^2+dy^2)^i}+\cdots
\end{gathered}
\end{equation}
Comparing the terms of the homogeneous degrees $-2$ and $-2i$ in \eqref{18101409} and \eqref{18111603}, we get  
\begin{equation}\label{18111606}
-\text{Im}\;c_1\dfrac{\pi^2}{|A|^2}=\dfrac{-\pi^2\text{Im}\;c_1}{x^2+dy^2}=\dfrac{a_1}{x^2+dy^2}
\end{equation} 
and 
\begin{equation}\label{18111604}
\begin{gathered}
\dfrac{1}{4}\dfrac{(-1)^i}{i}\text{Im}\Bigg(c_{2i-1}\Big(\dfrac{2\pi}{A}\Big)^{2i}\Bigg)=\dfrac{a_{i}}{(x^2+dy^2)^{i}}
\end{gathered}
\end{equation}
respectively. 
Since 
\begin{equation}\label{18102409}
\begin{gathered}
\dfrac{1}{4}\dfrac{(-1)^i}{i}\text{Im}\Bigg(c_{2i-1}\Big(\dfrac{2\pi}{A}\Big)^{2i}\Bigg)=\dfrac{(-1)^i(2\pi)^{2i}}{4i}\text{Im}\Bigg(\dfrac{c_{2i-1}\overline{A}^{2i}}{(A\overline{A})^{2i}}\Bigg)\\
=\dfrac{(-1)^i(2\pi)^{2i}}{4i}\text{Im}\Bigg(\dfrac{c_{2i-1}\overline{A}^{2i}}{(|A|^2)^{2i}}\Bigg)=\dfrac{(-1)^i(2\pi)^{2i}}{4i}\dfrac{\text{Im}\big(c_{2i-1}\overline{A}^{2i}\big)}{(x^2+dy^2)^{2i}}, 
\end{gathered}
\end{equation}
we get \eqref{18111606} and \eqref{18111604} are equivalent to 
\begin{equation*}
-\pi^2\text{Im}\; c_1=a_1
\end{equation*} 
and
\begin{equation}\label{18102410}
\dfrac{(-1)^i(2\pi)^{2i}}{4i}\dfrac{\text{Im}\big(c_{2i-1}\overline{A}^{2i}\big)}{(x^2+dy^2)^{2i}}=\dfrac{a_{i}}{(x^2+dy^2)^{i}}
\end{equation}
respectively. Now \eqref{18102410} implies
\begin{gather}
\dfrac{(-1)^i(2\pi)^{2i}}{4i}\dfrac{\text{Im}\big(c_{2i-1}\overline{A}^{2i}\big)}{(x^2+dy^2)^{i}}=a_{i}\nonumber\\
\Longrightarrow \dfrac{(-1)^i(2\pi)^{2i}}{4i}\text{Im}\big(c_{2i-1}\overline{A}^{2i}\big)=a_{i}(x^2+dy^2)^{i}\nonumber\\
\Longrightarrow \dfrac{(-1)^i(2\pi)^{2i}}{4i}\text{Im}\big(c_{2i-1}(x-\sqrt{-d}y)^{2i}\big)=a_{i}(x^2+dy^2)^{i}.\label{18102408}
\end{gather}
Comparing two expansions of $(x-\sqrt{-d}y)^{2i}$ and $(x^2+dy^2)^i$, it easily follows that there are no $c_{2i-1}\in \mathbb{C}, a_i\in \mathbb{R}$ satisfying \eqref{18102408} for two variables $x,y$.
\end{proof}

By Lemmas \ref{18112502} and \ref{18112601}, we get the following corollary:
\begin{corollary}\label{18102606}
Let $\mathcal{M}$, $\Theta_{\mathcal{M}}(x,y)$ and $Q_{\mathcal{M}}(x,y)$ be as above. Then 
\begin{equation}\label{18112904}
\Theta_{\mathcal{M}}(x,y)=\Theta_{\mathcal{M}}(z,w)
\end{equation}
is not equivalent to 
\begin{equation}\label{18112905}
\dfrac{1}{Q_{\mathcal{M}}(x,y)}=\dfrac{1}{Q_{\mathcal{M}}(z,w)}.
\end{equation}
In other words, \eqref{18112904} and \eqref{18112905} are different varieties. 
\end{corollary}

\subsection{Integer points on a $1$-dimensional definable set}
The lemma below is a simple application of Theorems \ref{18103002} and \ref{18103004}. In the lemma, we prove a uniform property concerning the number of integer points on a $1$-dimensional definable set in $\mathbb{R}_{an}$. 
\begin{lemma}\label{18110302}
Let $\mathcal{X}\subset \mathbb{R}^2(:=(x,y))$ be a $1$-dimensional definable set in $\mathbb{R}_{an}$ and $m,d(\in \mathbb{N})$ be fixed constants. Then there exists $c=c(\mathcal{X}, m, d)$ such that, for any $N$, the number of integer points on $\mathcal{X}$ satisfying
\begin{equation*}
N\leq x^2+dy^2< N+m 
\end{equation*} 
is uniformly bounded by $c$.  
\end{lemma}
\begin{proof}
Consider the following two $2$-dimensional definable sets in $\mathbb{R}^3$:
\begin{equation*}
\mathcal{X'}:=\{(x,y,z)\;|\;(x,y)\in \mathcal{X}\},
\end{equation*} 
\begin{equation*}
\mathcal{Y'}:=\{(x,y,z)\;|\;z=x^2+dy^2\}.
\end{equation*} 
Put
\begin{equation*}
\mathcal{Z'}:=\mathcal{X'}\cap \mathcal{Y'},
\end{equation*}
\begin{equation*}
\mathcal{Z'}_t:=\{(x,y)\;|\;(x,y,t)\in \mathcal{Z'}\},
\end{equation*}
and 
\begin{equation*}
\mathcal{Z'}(1):=\{t\;|\;\text{dim}\;\mathcal{Z'}_t=1\}.
\end{equation*}
Then $\mathcal{Z'}$ is a $1$-dimensional definable set and so $\mathcal{Z'}(1)$ is a definable set of dimension at most $0$ by Theorem \ref{18103004} (i.e., $\mathcal{Z'}(1)$ is a finite set). Let
\begin{equation*}
\mathcal{Z'}(1)=\{t_1,\dots, t_n\}.
\end{equation*}
By Theorem \ref{18103002}, there exists $C$ depending only on $Z'$ such that 
\begin{equation*}
|\mathcal{Z'}_{t}|<C
\end{equation*} 
for any $t\notin \mathcal{Z'}(1)$. Since the number of lattice points $N_{S_{t_i}}(\mathbb{Z})$ on $x^2+dy^2=t_i$ is finite for every $t_i\in \mathcal{Z'}(1)$, by letting 
\begin{equation*}
M=\text{max}\;\{C, N_{S_{t_1}}(\mathbb{Z}), \dots, N_{S_{t_n}}(\mathbb{Z})\},
\end{equation*} 
we get that, for any $t\in \mathbb{R}$, the number of integer points on
\begin{equation*}
\mathcal{Z'}_t=\{(x,y)\;|\;(x,y,t)\in \mathcal{Z'}\}
\end{equation*}
is uniformly bounded by $M$. Since 
\begin{equation*}
(x,y)\in \mathcal{Z'}_t \Longleftrightarrow (x,y)\in \mathcal{X},\quad x^2+dy^2=t,
\end{equation*}
we get the desired conclusion with $c=Mm$.  
\end{proof}


\section{Proof of the main theorem}\label{proof}
Now we prove our main theorem which we restate below. 

\begingroup
\def\thetheorem{\ref{18110705}}
\begin{theorem}
Let $\mathcal{M}$ be a $1$-cusped hyperbolic $3$-manifold whose cusp shape is quadratic. Then there exists $c=c(\mathcal{M})$ such that $N_{\mathcal{M}}(v)<c$ for any $v$. 
\end{theorem}
\addtocounter{theorem}{-1}
\endgroup

Let $\mathcal{Z}^k(\subset \mathbb{R}^2\times \mathbb{R}^2)$ and $\mathcal{Z}^k_{(x,y)}$ (for each $(x,y)\in \pi_1(\mathcal{Z}^k)$) be the same as \eqref{19032802} and \eqref{19032803} in Section \ref{Ph}. Note that, for each $k$, $\mathcal{Z}^k$ is a $2$-dimensional definable set in $\mathbb{R}_{an}$ by Corollary \ref{18102606}. 
 Also recall that, for a fixed $\mathcal{M}_{p/q}$, counting the number of Dehn fillings of $\mathcal{M}$ with volume $\text{vol}(\mathcal{M}_{p/q})$ is equivalent to counting the number of primitive integer points on $\bigcup_{-m< k< m} \mathcal{Z}^k_{(p,q)}$. 

Simply if $\mathcal{Z}^k_{(p,q)}$ is a zero dimensional definable set in $\mathbb{R}_{an}$ for any $(p,q)$, 
then $|\mathcal{Z}^k_{(p,q)}|$ is uniformly bounded by Theorem \ref{18103002} and so we get the desired result. Otherwise if $\mathcal{Z}^k_{(p,q)}$ is a $1$-dimensional definable set in $\mathbb{R}_{an}$, then a more careful analysis is required and a major part of the proof below will be devoted to deal with this situation. 

\begin{proof}
On the contrary, suppose there is an increasing sequence 
\begin{equation*}
n_1<n_2<n_3<\cdots
\end{equation*}
and a sequence of infinitely many Dehn fillings $\{M_{p_{ij}/q_{ij}}\}_{i\in \mathbb{N}, 1\leq j\leq n_i}$ such that
\begin{equation*}
\text{vol}(\mathcal{M}_{p_{ij}/q_{ij}})=\text{vol}(\mathcal{M}_{p_{ij'}/q_{ij'}})
\end{equation*}
for every $i\in \mathbb{N}$, $1\leq j,j'\leq n_i$. We further assume 
\begin{equation*}
\text{vol}(\mathcal{M}_{p_{11}/q_{11}})<\text{vol}(\mathcal{M}_{p_{21}/q_{21}})<\text{vol}(\mathcal{M}_{p_{31}/q_{31}})<\cdots.
\end{equation*}
By Lemma \ref{18110304}, for each $i$, there exists $N_i$ such that 
\begin{equation}\label{18111501}
N_i\leq p_{ij}^2+dq_{ij}^2< N_i+m\quad\quad (1\leq j\leq n_i),
\end{equation}
and thus every $(p_{ij},q_{ij}, p_{ij'},q_{ij'})$ (where $1\leq j, j'\leq n_i$) is contained in  
\begin{equation*}
\dfrac{1}{x^2+dy^2}=\dfrac{1}{z^2+dw^2+k}
\end{equation*}
for some $-m< k< m$.

Let $\mathcal{Z}^k, \mathcal{Z}^k_{(x,y)}$ be the same as above and denote the set of integer points in $\mathcal{Z}^k_{(x,y)}$ by $\mathcal{Z}^k_{(x,y)}(\mathbb{Z})$. Then the above hypothesis implies the following:\\
\\
\textit{$\bigstar$ For each $i$, there exist at least $n_i$ integral points $(p_{ij}, q_{ji})$ ($1\leq j\leq n_i$) in $\bigcup_{-m<k< m} \pi_1(\mathcal{Z}^k)$ such that 
\begin{enumerate}
\item $N_i\leq p_{ij}^2+dq_{ij}^2< N_i+m$\quad ($1\leq j\leq n_i$);
\item $\sum_{-m<k< m}|\mathcal{Z}^k_{(p_{ij}, q_{ij})}(\mathbb{Z})|>n_i$.  
\end{enumerate}
}
Now we prove 
\begin{claim}
Let $S=\{(x_l,y_l)\}$ be a sequence of integer points in $\pi_1(\mathcal{Z}^k)$. Then there exists $c$ such that, if 
\begin{equation*}
S_c:=\{(x_l, y_l)\;|\;|\mathcal{Z}^k_{(x_{l},y_{l})}(\mathbb{Z})|>c\},
\end{equation*}
there are at most $c$ elements in $S_c$ satisfying 
\begin{equation}
N_i\leq x_{l}^2+dy_{l}^2 < N_i+m
\end{equation} 
for each $i$.
\end{claim}

Assuming the claim, the rest of the proof goes as follows. Let $S=\{(x_l,y_l)\}$ be the set of all integer points in $\bigcup_{-m< k< m} \pi_1(\mathcal{Z}^k)$. By the claim, we find $c$ such that, for
\begin{equation*}
S_c:=\big\{(x_l, y_l)\;|\;\sum_{-m<k< m}|\mathcal{Z}^k_{(x_{l},y_{l})}(\mathbb{Z})|>c\big\},
\end{equation*}
there are at most $c$ elements in $S_c$ satisfying 
\begin{equation}
N_i\leq x_{l}^2+dy_{l}^2 < N_i+m
\end{equation} 
for each $i$. But this contradicts the statement made in $\bigstar$, since $\{n_i\}$ is an increasing sequence and so $n_i>c$ for some $i$.\\
\\
\textit{Proof of the claim}
Let 
\begin{equation*}
\{\mathcal{Z}^k_i\;|\;i\in \mathcal{I}\}
\end{equation*}
be the set of all the connected components of $\mathcal{Z}^k$ and $\mathcal{W}$ be an element in $\{\mathcal{Z}^k_{i}\;|\;i\in \mathcal{I}\}$. Since a definable set has only finitely many connected components (i.e. $\mathcal{I}$ is a finite set), to prove the claim, it is enough to show it for $\mathcal{W}$. 

Now we split the problem into two cases depending on dim $\mathcal{W}$. 
\begin{enumerate}
\item \text{dim} $\mathcal{W}=2$\\
We further divide the problem into the following two subcases:
\begin{enumerate}
\item If
\begin{equation*}
\text{dim }\mathcal{W}_{(x,y)}=0
\end{equation*}
(i.e., $\mathcal{W}_{(x,y)}$ is a finite set) for any $(x,y)\in \pi_1\big(\mathcal{W}(\mathbb{Z})\big)$, 
then, by Theorem \ref{18103002}, there exists $c$ such that 
\begin{equation*}
|\mathcal{W}_{(x,y)}(\mathbb{Z})|<c
\end{equation*}
for any $(x,y)\in \pi_1\big(\mathcal{W}(\mathbb{Z})\big)$. Thus in this case, $S_c$ is the empty set and so the assertion of the claim is true. 

\item  Now suppose there exists $(x, y)\in \pi_1\big(\mathcal{W}(\mathbb{Z})\big)$ such that
\begin{equation}\label{18110301}
\text{dim }\mathcal{W}_{(x,y)}=1. 
\end{equation}
Let
\begin{equation*}
\mathcal{W}(1):=\{(x,y)\in \pi_1(\mathcal{W})\;:\;\text{dim}\;\mathcal{W}_{(x,y)}=1\}.
\end{equation*}
By Theorem \ref{18103004}, $\mathcal{W}(1)$ is a definable set in $\mathbb{R}_{an}$ and its dimension is at most $1$. 
\begin{enumerate}
\item Suppose dim $\mathcal{W}(1)=0$ and put
\begin{equation*}
\mathcal{W}(1)=\{(x_{i_1},y_{i_1}),\dots, (x_{i_k}, y_{i_k})\}.
\end{equation*}
By Theorem \ref{18103002}, there exists a universal constant $c'$ such that 
\begin{equation*}
|\mathcal{W}_{(x,y)}|<c'
\end{equation*}
for any $(x,y)\notin \mathcal{W}(1)$. Since $\mathcal{W}_{(x_i,y_i)}(\mathbb{Z})$ is a finite set for every $(x_i, y_i)\in \mathcal{W}(1)$, by letting 
\begin{equation*}
c=\max_{1\leq j\leq k}\,\{|\mathcal{W}_{(x_{i_j},y_{i_j})}(\mathbb{Z})|, c'\},
\end{equation*}
we again get $S_c$ is the empty set and so obtain the desired conclusion. 

\item Now assume dim $\mathcal{W}(1)=1$. First, by Theorem \ref{18103002}, there exists $c'$ such that
\begin{equation*}
|\mathcal{W}_{(x,y)}|<c'
\end{equation*}
for any $(x,y)\notin \mathcal{W}(1)$. Then the following 
\begin{equation*}
S_{c'}:=\big\{(x,y)\in S\;|\;|\mathcal{W}_{(x,y)}(\mathbb{Z})|\geq c'\big\}
\end{equation*}
is a subset of $\mathcal{W}(1)$. By Lemma \ref{18110302}, there exists $c''$ such that, for any $N$, the number of integer points on $\mathcal{W}(1)$ satisfying 
\begin{equation*}
N\leq x^2+dy^2< N+m
\end{equation*}
is uniformly bounded by $c''$. By letting 
\begin{equation*}
c=\max\,\{c',c''\},
\end{equation*}
we get the desired result. 
\end{enumerate}
\end{enumerate}

\item \text{dim} $\mathcal{W}=1$
\begin{enumerate}
\item If dim $\mathcal{W}_{(x,y)}=0$ for any $(x,y)\in \pi_1\big(\mathcal{W}(\mathbb{Z})\big)$, then, by Theorem \ref{18103002}, we find $c$ satisfying 
\begin{equation*}
|\mathcal{W}_{(x,y)}|<c
\end{equation*}
for every $(x,y)\in \pi_1\big(\mathcal{W}(\mathbb{Z})\big)$. So $S_c$ is the empty set and we get the desired result.
\item Suppose there exists $(x,y)\in \pi_1\big(\mathcal{W}(\mathbb{Z})\big)$ such that dim $\mathcal{W}_{(x,y)}=1$ and put
\begin{equation*}
\mathcal{W}(1):=\{(x,y)\in \pi_1(\mathcal{W})\;|\;\text{dim}\;\mathcal{W}_{(x,y)}=1\}.
\end{equation*}
By Theorem \ref{18103004}, $\mathcal{W}(1)$ is a definable set of dimension $0$. If
\begin{equation*}
\mathcal{W}(1)=\{(x_{i_1}, y_{i_1}), \dots, (x_{i_k}, y_{i_k})\},
\end{equation*}
then $\mathcal{W}_{(x_i,y_i)}(\mathbb{Z})$ is a finite set for each $(x_i, y_i)\in \mathcal{W}(1)$. By Theorem \ref{18103002}, there exists $c'$ such that 
\begin{equation*}
|\mathcal{W}_{(x,y)}|<c'
\end{equation*}
for every $(x,y)\notin \mathcal{W}(1)$. Thus, by letting 
\begin{equation*}
c:=\max_{1\leq j\leq k}\,\big\{|\mathcal{W}_{(x_{i_j},y_{i_j})}(\mathbb{Z})|, c'\big\},
\end{equation*}
we get that every point $(x,y)$ in $\pi_1\big(\mathcal{W}(\mathbb{Z})\big)$ satisfies
\begin{equation*}
|\mathcal{W}_{(x,y)}(\mathbb{Z})|<c.
\end{equation*}  
In conclusion, $S_c$ is again the empty set and this finishes the proof. 
\end{enumerate}
\end{enumerate}
\end{proof}

\section{Conjectures}\label{18110807}
In this section, we propose several conjectures concerning the main question (i.e. Question \ref{18102801}). 

Let us first begin with some basic definitions in number theory.

\begin{definition}
For $\alpha:=(a_1/b_1,\dots, a_n/b_n)\in \mathbb{Q}^n$ where $\gcd\,(a_i, b_i)=1$ for each $i$, we define the height of it as 
\begin{equation*}
H(\alpha):=\max_{1\leq i\leq n}\,\{|a_i|, |b_i|\}.
\end{equation*}
\end{definition}

\begin{definition}
For $X\subset \mathbb{R}^n$, let $X(\mathbb{Q})$ denote the subset of points with rational coordinates and set 
\begin{equation*}
X(\mathbb{Q}, T):=\{P\in X(\mathbb{Q}), \quad H(P)\leq T\}.
\end{equation*}
Define the density function of $X$ to be 
\begin{equation*}
N(X, T):=|X(\mathbb{Q}, T)|.
\end{equation*}
\end{definition}

In \cite{BP}, Bombieri-Pila studied the growth rate of $N(X,T)$ when $X$ is the graph of a  transcendental function. In particular, they showed that it grows fairly slowly and so rational points on $X$ are generally very sparse. 
\begin{theorem}[Bombieri-Pila]\label{18102601}
Let $f\;:\;[0,1]\longrightarrow \mathbb{R}$ be a transcendental function, and $X$ be the graph of $f$. Then, for any $\epsilon>0$, there exists $c=c(X, \epsilon)$ such that 
\begin{equation*}
N(X, T)\leq cT^{\epsilon}.
\end{equation*}  
\end{theorem}
For instance, if $X$ is defined by $y=\dfrac{1}{a^x}$ where $a\in \mathbb{N}$, then the density function grows logarithmically and so it satisfies the assertion of the theorem. 

To generalize the above theorem on $X$ of dimension $2$, similar to the the previous case, we first need to assume $X$ itself is not semialgebraic. 

Now consider  
\begin{equation*}
X:=\{(x,y,z)\;:\;z=x^y, \quad x, y\in [1,2]\}.
\end{equation*}
In the example, $X$ itself is clearly not semialgebraic (otherwise $z=a^x$ would be algebraic for $a\in [1,2]$), but each $y\in \mathbb{Q}$ gives rise to a semialgebraic curve embedded in $X$ containing many rational points. Thus, to accommodate this feature, we make the following definition. 
\begin{definition}
Let $X\subset \mathbb{R}^n$. The algebraic part of $X$, denoted $X^{alg}$, is the union of all connected semialgebraic subsets of $X$ of positive dimension. The transcendental par of $X$ is the complement $X-X^{alg}$, and denoted by $X^{trans}$. 
\end{definition}

Theorem \ref{18102601} was further generalized to surfaces by Pila as follows \cite{P}: 

\begin{theorem}[Pila]\label{18102602}
Let $M$ be a compact subanalytic surface. For any $\epsilon>0$, there exists $c=c(M, \epsilon)$ such that 
\begin{equation*}
N(X^{trans}, T)\leq cT^{\epsilon}.
\end{equation*}  
\end{theorem}
Finally the complete generalization was obtained by Pila-Wilkie using o-minimality \cite{PW}.  

\begin{theorem}[Pila-Wilkie]\label{18102603}
Let $\chi$ be an o-minimal structure and $X$ be a definable set in $\chi$. Then, for any $\epsilon$, there exists $c=c(X, \epsilon)$ such that 
\begin{equation*}
N(X^{trans},T)\leq c(Z, \epsilon)T^{\epsilon}.
\end{equation*}
\end{theorem}

Now a strategy for answering Question \ref{18102801} using the above theorem goes as  follows. As usual, we denote a $1$-cusped hyperbolic $3$-manifold and its volume function by $\mathcal{M}$ and $\Theta_{\mathcal{M}}(x,y)$ respectively. Let $\mathcal{Z}$ be an analytic variety defined by
\begin{equation}\label{18102803}
\Theta_{\mathcal{M}}(x,y)=\Theta_{\mathcal{M}}(z,w).
\end{equation}
By Theorem \ref{18102603}, most rational points of $\mathcal{Z}$ are lying on the semialgebraic subsets of $\mathcal{Z}$ and thus if 
\begin{equation*}
\text{vol}(\mathcal{M}_{p/q})=\text{vol}(\mathcal{M}_{p'/q'}), 
\end{equation*}
it is likely that $(p,q,p',q')$ is contained in an semialgebraic subset of $\mathcal{Z}$. An obvious semialgebraic set of $\mathcal{Z}$ is the surface defined by 
\begin{equation*}
x=z, \quad y=w.
\end{equation*}
In general, we do not know how to prove it but, based on our study, propose the following conjecture:

\begin{conjecture}
Let $\mathcal{M}$ and $\mathcal{Z}$ be as above. Then any semialgebraic subset of $\mathcal{Z}$ is contained in 
\begin{equation*}
z=ax+by, \quad w=cx+dy
\end{equation*}
for some $a,b,c,d\in \mathbb{Z}$ satisfying
\begin{equation}\label{19032805}
Q_{\mathcal{M}}(x,y)=Q_{\mathcal{M}}(ax+by, cx+dy).
\end{equation}
\end{conjecture}

Since there any only finitely many $(a,b,c,d)\in \mathbb{Z}^4$ satisfying \eqref{19032805}, the above conjecture, combining with the Pila-Wilkie theorem, implies the following conjecture, which answers Question \ref{18102801} affirmatively for $n=1$.
\begin{conjecture}
Let $\mathcal{M}$ be a $1$-cusped hyperbolic $3$-manifold. Then there exists 
\begin{equation*}
\{(a_i,b_i,c_i,d_i)\in \mathbb{Z}^4\}_{1\leq i\leq k} 
\end{equation*}
such that if
\begin{equation*}
\text{\normalfont\ vol}(\mathcal{M}_{p/q})=\text{\normalfont\ vol}(\mathcal{M}_{p'/q'}),
\end{equation*}
then 
\begin{equation*}
p=a_ip'+b_iq', \quad q=c_ip'+d_iq'
\end{equation*}
for some $1\leq i\leq k$. 
\end{conjecture}

More generally, we expect a similar phenomenon is true for any $n$-cusped hyperbolic $3$-manifold. 

\begin{conjecture}
Let $\mathcal{M}$ be an $n$-cusped hyperbolic $3$-manifold. Then there exists 
\begin{equation*}
\{(a_{i,j},b_{i,j},c_{i,j},d_{i,j})\in \mathbb{Z}^4)\}_{(1\leq i\leq n, 1\leq j\leq n_i)}
\end{equation*}
such that if
\begin{equation*}
\text{\normalfont\ vol}\big(\mathcal{M}_{(p_1/q_1, \dots, p_n/q_n)}\big)=\text{\normalfont\ vol}\big(\mathcal{M}_{(p'_1/q'_1, \dots, p'_n/q'_n)}\big),
\end{equation*}
then
\begin{equation*}
p_i=a_{\sigma(i),j}p'_{\sigma(i)}+b_{\sigma(i),j}q'_{\sigma(i)}, \quad q_i=c_{\sigma(i),j}p'_{\sigma(i)}+d_{\sigma(i),j}q'_{\sigma(i)}
\end{equation*}
for some $\sigma\in S_n$ and $1\leq j\leq n_{\sigma(i)}$. 
\end{conjecture}

Finally the above conjecture implies the following conjecture, which answers Question \ref{18102801} as well as Question \ref{gromov} completely.

\begin{conjecture}
Let $\mathcal{M}$ be an $n$-cusped hyperbolic $3$-manifold. Then there exists $c=c(\mathcal{M})$ such that the number of hyperbolic Dehn fillings of $\mathcal{M}$ with any given volume $v$ is uniformly bounded by $c$. 
\end{conjecture}

\vspace{5 mm}
Department of Mathematics, POSTECH \\
77 Cheong-Am Ro, Pohang, South Korea\\
\\
\emph{Email Address}: bogwang.jeon@postech.ac.kr

\end{document}